\documentclass[11pt]{amsart}

\usepackage[]{amsmath,amsxtra,amssymb,latexsym, amscd,amsthm, amsfonts,tikz}
\usepackage{hyperref}
\hypersetup{
  colorlinks   = true, 
  urlcolor     = blue, 
  linkcolor    = blue, 
  citecolor   = red 
}
\usepackage{graphicx}
\usepackage{multicol}
\usepackage{tikz}
\usepackage[subnum]{cases}
\usepackage{multirow}
\usepackage{enumerate}
\usepackage{caption}
\usetikzlibrary{arrows,decorations.markings}
 \usetikzlibrary{calc} 
\usetikzlibrary{matrix}
\usetikzlibrary{graphs}
\usetikzlibrary{backgrounds}
\def\DynkinNodeSize{2mm}
\def\DynkinArrowLength{3mm}
\tikzset{
  dnode/.style={
    circle,
    inner sep=0pt,
    minimum size=\DynkinNodeSize,
    fill=white,
    draw},
  middlearrow/.style={
    decoration={markings,
      mark=at position 0.6 with
      {\draw (0:0mm) -- +(+135:\DynkinArrowLength); \draw (0:0mm) -- +(-135:\DynkinArrowLength);},
    },
    postaction={decorate}
  },
  leftrightarrow/.style={
    decoration={markings,
      mark=at position 0.999 with
      {
      \draw (0:0mm) -- +(+135:\DynkinArrowLength); \draw (0:0mm) -- +(-135:\DynkinArrowLength);
      },
      mark=at position 0.001 with
      {
      \draw (0:0mm) -- +(+45:\DynkinArrowLength); \draw (0:0mm) -- +(-45:\DynkinArrowLength);
      },
    },
    postaction={decorate}
  },
  sedge/.style={
  },
  dedge/.style={
    middlearrow,
    double distance=1mm,
  },
  tedge/.style={
    middlearrow,
    double distance=1.0mm+\pgflinewidth,
    postaction={draw}, 
  },
  infedge/.style={
    leftrightarrow,
    double distance=0.5mm,
  },
}

\newtheorem{theorem}{Theorem}[section]
\newtheorem{lemma}[theorem]{Lemma}

\theoremstyle{definition}

\theoremstyle{remark}
\newtheorem{remark}[theorem]{Remark}

\theoremstyle{proposition}
\newtheorem{proposition}[theorem]{Proposition}

\theoremstyle{corollary}
\newtheorem{corollary}[theorem]{Corollary}
 
\numberwithin{equation}{section}

\newcommand{\vn}{\noindent}
\newcommand{\R}{\mathbb{R}}

\newcommand{\Z}{\mathbb{Z}}

\newcommand{\D}{{\mathcal{D}}}

\newcommand{\U}{{\mathcal{U}}}

\newcommand{\DP}{{\mathcal{DP}}}
 \newcommand{\lan}{\langle}
\newcommand{\ran}{\rangle}

\begin{document}

\title[The largest coefficient and the second smallest exponent]{The largest coefficient of the highest root and the second smallest exponent} 
\author{Tan Nhat Tran}

\address{Department of Mathematics, Hokkaido University, Sapporo 060-0810, Japan}

\email{trannhattan@math.sci.hokudai.ac.jp}

\subjclass[2010]{17B22 (Primary), 05A19 (Secondary)}

\date{\today}

\keywords{Root system, highest root, Weyl group, exponent}

\begin{abstract}

There are many different ways that the exponents of Weyl groups of irreducible root systems have been defined and put into practice. One of the most classical and algebraic definitions of the exponents is related to the eigenvalues of Coxeter elements. While the coefficients of the height root when expressed as a linear combination of simple roots are combinatorial objects in nature, there are several results asserting relations between these exponents and  coefficients. This study was conducted to give a uniform and fairly elementary proof of the fact that the second smallest exponent of the Weyl group is one or two plus the largest coefficient of the highest root of the root system depending upon a simple condition on the root lengths. As a consequence, we obtain a necessary and sufficient condition for a root system to be of type $G_2$ in terms of these numbers.

\end{abstract}

\maketitle

\section{Introduction}

Assume that $V=\R^\ell$ with the standard inner product $(\cdot,\cdot)$. 
For $\alpha \in V$, $\beta \in V\setminus\{0\}$,  denote $\langle \alpha,\beta \rangle := \frac{2( \alpha,\beta ) }{( \beta , \beta)}$. 
Let us denote by $\Phi$ an irreducible crystallographic root system in $V$. 
Let $\Phi^+$ be a set of positive roots.
With the notation $\Delta=\{\alpha_1, \ldots ,\alpha_\ell\}$, we have the simple root system of $\Phi$ with respect to $\Phi^+$. 
For any $\alpha, \beta \in \Phi$, the number $\langle \alpha,\beta \rangle$ takes values in $\{0, \pm1,\pm2,\pm3\}$.
For $\alpha = \sum_{i=1}^\ell  d_i \alpha_i \in \Phi^+$, the \textit{height} of $\alpha$ is defined by $ {\rm ht}(\alpha) :=\sum_{i=1}^\ell  d_i$. 
Define the partial order $\le$ on $\Phi^+$ such that $\beta \le \alpha$ if $\alpha-\beta \in \sum _{i=1}^\ell  \Z_{\ge0} \alpha_i$ for $\alpha, \beta\in \Phi^+$.
Let $\theta:= \sum_{i=1}^\ell  c_i\alpha_i$ be the highest root of $\Phi$ with respect to the partial order, and we call $c_i$'s the coefficients of $\theta$. 
Denote by $c_{\max}:=\max \{c_i\mid1 \le i\le \ell\}$ the largest coefficient. 
Let $W$ be the Weyl group of $\Phi$ and let $m_1, m_2, \ldots, m_\ell$ with $m_1 \le m_2 \le \ldots \le m_\ell$ be the exponents of $W$.

The exponents of the Weyl group may have been originally defined in terms of the eigenvalues of Coxeter elements \cite{C51}. 
In addition, they can be defined as the degrees of the basic polynomial invariants of the Weyl group \cite{C55}. 
The multiset of the exponents and its subsets also have led to many important  results and applications in study of Weyl arrangements, which are important examples of free arrangements (\cite{OS83}, \cite{OST86}, \cite[Chapter 6]{OT92}).
All of these above-mentioned definitions and applications are purely algebraic.
Shapiro (unpublished), Steinberg \cite{R59}, Kostant \cite{K59}, Macdonald \cite{M72} and most recently also Abe-Barakat-Cuntz-Hoge-Terao \cite{ABCHT16} have found and shown that there is another possibility to obtain the exponents, namely the dual partition of the height distribution of $\Phi^+$. 
This latter approach not only gives a particularly simple way of determining the exponents in the individual cases but also reveals  connections between the exponents and the other combinatorial objects of the root system.
There are also many results in the literature asserting relations between the exponents and the coefficients of the highest root.
The most fundamental one is known that the largest exponent is equal to the sum of the highest root coefficients i.e., $m_\ell = \sum_{i=1}^\ell  c_i$.
A complete description of the exponents in terms of $c_{\max}$ found by a case-by-case check is mentioned in \cite[Theorem 3.2]{Bu09}.
What most interests us is the following interesting relation between $c_{\max}$ and $m_2$, which we describe in a uniform way. 
\begin{theorem}\label{thm:uniform} 
Assume that $\ell\ge 2$. 
Set $\U:=\{\theta_i\in \Phi^+ \mid {\rm ht}(\theta_i) >m_{\ell-1}\}$,
and $m:=|\U|$. 
Suppose that $\xi_i:=\theta_i-\theta_{i+1} \in\Delta$ for $1 \le i \le m-1$.
If there is an integer $t$ such that $1 \le t \le m-1$ and $\lan\theta_t,\xi_t \ran= 3$, then $c_{\max}=m_2-2$. For otherwise, $c_{\max}=m_2-1$.
\end{theorem}

By a ``uniform" way we mean the proof does not rely on the Classification Theorem of root systems \cite[Chapter III, 11.4, Theorem]{H72} except the fact that the Dynkin graph of $\Phi$ is a tree. 
As a consequence, we obtain a criterion for a root system to be of type $G_2$ in terms of $c_{\max}$ and $m_2$.
\begin{corollary}\label{cor:criterion}
$\Phi$ is of type $G_2$ if and only if $c_{\max}=m_2-2$.
\end{corollary}

The aim of this paper is to provide a uniform and fairly elementary proof of Theorem \ref{thm:uniform}.
We build two isomorphic sets whose cardinalities are expressed in terms of $c_{\max}$ and $m_2$, respectively. 
One is described by a graph-theoretical property of the Dynkin graph, while the other is described by a combinatorial property of the root poset. 
The isomorphism between these sets is proved according to the condition on the root lengths (Theorem \ref{thm:iso}). 

The rest of this paper is organized as follows. 
In \S\ref{sec:pre} we review some fundamental definitions and results about root systems and Weyl groups. 
\S\ref{sec:properties} is intended to motivate our investigation on combinatorial and graph-theoretical properties of positive roots. 
\S\ref{sec:proof} contains  the proofs of Theorem \ref{thm:uniform} and Corollary \ref{cor:criterion}.
\section{Preliminaries}\label{sec:pre}
 
Our standard references for root systems and their Weyl groups are \cite{B68} and \cite[Chapter III]{H72}.
Let $V= \R^\ell $.
Let $\Phi$ be an irreducible (crystallographic) root system spanning $V$ with the standard inner product $(\cdot,\cdot)$. 
We fix a positive system $\Phi^+$ of $\Phi$.
We write $\Delta:=\{\alpha_1, \ldots ,\alpha_\ell\}$ for the simple system (base) of $\Phi$ with respect to $\Phi^+$. 
For $\alpha \in V$,  denote $\|\alpha \|:=\sqrt{(\alpha ,\alpha)}$. 
Note that at most two root lengths can occur in $\Phi$ \cite[Chapter III, 10.4, Lemma C]{H72}.

A \emph{reflection} in $V$ with respect to a nonzero vector $\alpha \in V$ is a mapping $s_{\alpha}: V \to V$ defined by $s_\alpha (x) := x -\langle x,\alpha\rangle \alpha$.
The \emph{Weyl group} $W:=W(\Phi)$ of $\Phi$ is a group generated by the set $\{s_{\alpha}\mid \alpha \in \Phi\}$.
 An element of the form $c=s_{\alpha_1}\dots s_{\alpha_\ell}\in W$ is called a \emph{Coxeter element}. 
 Since all Coxeter elements are conjugate \cite[Chapter V, $\S$6.1, Proposition 1]{B68}, they have the same order, characteristic polynomial and eigenvalues. 
The order ${\rm h}:={\rm h}(W)$ of Coxeter elements is called the \emph{Coxeter number} of $W$.
For a fixed Coxeter element $c\in W$, if its eigenvalues are of the form $\exp (2\pi\sqrt{-1}m_1/{\rm h}),\ldots, \exp (2\pi\sqrt{-1}m_\ell/{\rm h})$ with $0< m_1 \le \ldots \le m_\ell<{\rm h}$, then the integers  $m_1,\ldots, m_\ell$ are called the \emph{exponents} of $W$. 
The following facts can be found in \cite[Chapter V, $\S$6.2 and Chapter VI, $\S$1.11]{B68}.
\begin{theorem}\label{exponents} 
For any irreducible root system $\Phi$ in $\R^\ell$,
\begin{enumerate}[(i)]
\item    $m_j + m_{\ell+1-j}={\rm h}$ for $1 \le j \le \ell$,
\item   $1=m_1 < m_2 \le \ldots \le m_{\ell-1} <m_\ell={\rm h}-1$,
\item ${\rm h}={\rm ht}(\theta)+1$, where $\theta$ is the highest root.
\end{enumerate}
\end{theorem}

Let $\Theta^{(r)} \subseteq \Phi^+$ be the set consisting of positive roots of height $r$.
The \textit{height distribution} of $\Phi^+$ is defined as a multiset  of positive integers:
$$\{t_1, \ldots , t_r, \ldots , t_{{\rm h}-1}\},$$ where 
$t_r := \left|\Theta^{(r)}\right|$.
The  \textit{dual partition} $\DP(\Phi^+)$ of the height distribution of $\Phi^+$ is given by a multiset of nonnegative integers:
$$\DP(\Phi^+) := \{(0)^{\ell-t_1},(1)^{t_1-t_2},\ldots ,({\rm h}-2)^{t_{{\rm h}-2}-t_{{\rm h}-3}},({\rm h}-1)^{t_{{\rm h}-1}}\},$$ 
where notation $(a)^b$ means the integer $a$ appears exactly $b$ times.
 \begin{theorem}[\cite{R59},  \cite{K59}, \cite{M72}, \cite{ABCHT16}]\label{thm:dual}
The exponents of the Weyl group are given by $\DP(\Phi^+)$.
\end{theorem}

\section{Graph-theoretical and combinatorial properties of roots}\label{sec:properties}
In the remainder of the paper, we assume that $\ell\ge 2$.
We denote by $\D(\Phi)$ the \emph{Dynkin graph} and by $\widetilde{\D}(\Phi)$ the extended Dynkin graph of $\Phi$. 
A vertex of a graph is called a \emph{terminal} vertex (resp., a \emph{ramification point}) if it is adjacent to at most one other vertex (resp., to at least three other vertices). 
A  graph is a \emph{simple chain of length $n\ge0$} if it is isomorphic to the Dynkin graph of a root system of type $A_{n+1}$.
Two adjacent vertices $\alpha, \beta$ of $\widetilde{\D}(\Phi)$ are joined by a single (resp., double or triple) edge if $\|\alpha\|=\|\beta\|$ (resp., $\|\alpha\|=\sqrt2\|\beta\|$ or $\|\alpha\|=\sqrt3\|\beta\|$).

We start with a construction of a set whose cardinality is equal to  $c_{\max}$.
It is inspired by a graph-theoretical interpretation \cite[Lemma B.27, Appendix B]{MT11} of the highest root coefficients and the associated simple roots, which was originally formulated and proved in terms of coroots in  \cite[Lemma 1.5]{R75}. 
Other related results can be found in \cite[Proposition 2.1]{Bu07}. 
There was an unfortunate error in the proof of \cite[Lemma B.27, Appendix B]{MT11} and the proof itself was not completely correct. However, arguments used there can be well carried to restate the result correctly.
We provide here a detailed edition for the reader's convenience.
\begin{proposition}\label{lem:coes} 
Let $\Phi$ be an irreducible root system in $\R^\ell$. 
Let  $\theta$ be the highest root of $\Phi$, and denote $\lambda_0 := -\theta$, $c_{\lambda_0}:=1$.
Suppose that the elements of a fixed base $\Delta:=\{\lambda_1, \ldots, \lambda_\ell\}$ are labeled so that $\Lambda:=\{\lambda_0, \lambda_1, \ldots, \lambda_q\}$ is a set of minimal cardinality such that $c_{\max}=c_{\lambda_q}$ and $(\lambda_s,\lambda_{s+1}) < 0$ for $0 \le s \le q -1$. 
\begin{enumerate}[(i)] 
\item Then $c_{\lambda_s}=s + 1$ for $0\le s \le q$ and $|\Lambda|=c_{\max}$.
\item Assume that  $c_{\max}\ge 2$. Then $(\lambda_0,\lambda_1, \ldots, \lambda_{q-1})$ is a simple chain of $\widetilde{\D}(\Phi)$ connected to the other vertices only at $\lambda_{q-1}$. 
\end{enumerate}
\end{proposition}
\begin{proof}
If $c_{\max}=1$, obviously, $\Lambda=\{\lambda_0\}$.
Now assume that $c_{\max}\ge 2$.
\begin{equation*}\label{many}
\begin{aligned}
2=\lan\theta,\theta\ran & = \sum_{s=1}^\ell  c_{\lambda_s}\lan\lambda_s,\theta\ran \ge c_{\lambda_1}\lan\lambda_1,\theta\ran, \\
\lan\theta,\lambda_1\ran & = \sum_{s=1}^\ell  c_{\lambda_s}\lan\lambda_s,\lambda_1\ran \le 2c_{\lambda_1} +c_{\lambda_2}\lan\lambda_2,\lambda_1\ran,  \\
0 \le \lan\theta,\lambda_j\ran & \le c_{\lambda_{j-1}}\lan\lambda_{j-1},\lambda_j\ran+2c_{\lambda_j} +c_{\lambda_{j+1}}\lan\lambda_{j+1},\lambda_j\ran \quad (2 \le j \le q-1).
\end{aligned}
\end{equation*}
By definition of $\Lambda$, $\lan\lambda_1,\theta\ran=1$. Thus 
\begin{equation}\label{equas}
\begin{aligned}
2-c_{\lambda_1} &  \ge 0, \\
\quad 2c_{\lambda_1} -c_{\lambda_2}-   \lan\theta,\lambda_1\ran &  \ge 0,\\
- c_{\lambda_{j-1}}+2c_{\lambda_j} -c_{\lambda_{j+1}} & \ge 0.
\end{aligned}
\end{equation} 

Adding up the inequalities in \eqref{equas} yields
$$2-   \lan\theta,\lambda_1\ran  \ge c_{\lambda_q}-c_{\lambda_{q-1}} .$$

\vn
If $\lan\theta,\lambda_1\ran=2$, by the minimality, we must have $q=1$, $\Lambda=\{\lambda_0, \lambda_1\}$, and $c_{\max}=c_{\lambda_1}=2$.

\vn
If $\lan\theta,\lambda_1\ran=1$, by the minimality, $1+ c_{\lambda_{q-1}}= c_{\lambda_q}$. 
Thus equality occurs here and also in each of the inequalities used above. 
We obtain a recurrence relation defined by $c_{\lambda_0}=1$, $c_{\lambda_1}=2$, $c_{\lambda_{j+1}}= 2c_{\lambda_j} -c_{\lambda_{j-1}}$ $(1 \le j \le q-1)$.
Thus $c_{\lambda_s}=s + 1$ for $0 \le s\le q$.
Additionally, from $\lan\lambda_{j-1},\lambda_j\ran=\lan\lambda_{j},\lambda_{j-1}\ran=-1$ $(1 \le j \le q-1)$, we get
$\|\lambda_0\|=\|\lambda_1\|=\ldots=\|\lambda_{q-1}\|$.
Thus $(\lambda_0,\lambda_1, \ldots, \lambda_{q-1})$ is a simple chain of $\widetilde{\D}(\Phi)$ connected to the other vertices only at $\lambda_{q-1}$. 
\end{proof}
\begin{remark}\label{rem:only1}
If $c_{\max}=1$, $\D(\Phi)$ contains only single edges (i.e., all roots of $\Phi$ have the same length). 
In addition, if $\ell \ge 2$, $-\theta$ is connected only to the terminal vertices of $\D(\Phi)$. 
Furthermore, the equation $\lan\theta,\theta\ran=2$ implies that $\D(\Phi)$ has exactly two terminal vertices. 
In this case, we know explicitly that $\D(\Phi)$ must be a simple chain.
 If $c_{\max}\ge 2$, by Proof of Proposition \ref{lem:coes}, $-\theta$ is connected only to one vertex of ${\D}(\Phi)$.
\end{remark}
\begin{corollary}\label{cor:not-important}
Assume that $c_{\max}\ge 2$. Either $\lan\lambda_{q-1},\lambda_{q}\ran\in\{-2,-3\}$ or $\lambda_{q}$ is a ramification point of $\widetilde{\D}(\Phi)$.
\end{corollary}
\begin{proof}
Assume that  $c_{\max}= 2$ i.e., $q=1$. Suppose that $\lan\lambda_0,\lambda_1\ran=-1$, and $\lambda_1$ is connected only to one vertex of $\widetilde{\D}(\Phi)$ apart from $\lambda_0$, say $\lambda_2$. 
Thus $\lambda_1$ is long and $\lan\lambda_2,\lambda_1\ran=-1$. 
From $\lan\lambda_0,\lambda_1\ran=-1$, we get $c_{\lambda_2}=3$, which is absurd. 
The case $c_{\max}\ge 3$ i.e., $q \ge 2$ is treated similarly by using $\lan\lambda_0,\lambda_q\ran=0$ in place of $\lan\lambda_0,\lambda_1\ran=-1$.
\end{proof}

Next, we prove several combinatorial properties of positive roots according to their locations on the root poset (with respect to the partial order $\le$).
\begin{lemma}\label{rem:k=2} 
Assume that $\beta \in \Phi^+$, $\alpha \in \Delta$, and $\lan\beta, \alpha \ran=k \in \{2,3\}$. 
Then there exists $\alpha' \in \Delta\setminus\{\alpha\}$ such that $\beta- (k-1)\alpha -\alpha'\in \Phi^+$. 
\end{lemma}
\begin{proof}
By the assumption, $s_\alpha (\beta)= \beta -k \alpha \in \Phi$.
Thus ${\rm ht}(\beta)\ge k+1$ and $\| \beta \| = \| \beta -k \alpha \| > \| \alpha \|$. 
In addition, $\beta- i\alpha \in \Phi^+$ for all $0 \le i\le k$ since the $\alpha$-string through $\beta$ is unbroken \cite[Chapter III, 9.4]{H72}.
If $\lan\beta- (k-1)\alpha, \alpha \ran \ge 1$, then $\|\beta -k \alpha\| \le \|  \beta -(k-1) \alpha\|$ i.e., $ \beta -(k-1) \alpha$ is a long root. 
We then have $\lan\beta- (k-1)\alpha, \alpha \ran \ge 2$. 
But it implies that  $\|\beta -k \alpha\| < \|  \beta -(k-1) \alpha\|$, a contradiction.
Thus $(\beta- (k-1)\alpha, \alpha)\le 0$. 
Suppose that $(\beta-(k-1)\alpha, \alpha' )\le 0$  for all $ \alpha' \in \Delta\setminus\{\alpha\}$. 
By \cite[Chapter III, 10.1, Theorem$^\prime$(3)]{H72}, $\{\beta- (k-1)\alpha\}\cup \Delta$ is a linearly independent set, which is absurd.
There  exists $\alpha' \in \Delta\setminus\{\alpha\}$ such that  $(\beta- (k-1)\alpha, \alpha')> 0$ hence $\beta- (k-1)\alpha -\alpha' \in \Phi^+$. 
\end{proof}

\begin{lemma}\label{lem:3roots} 
Suppose $\beta_1,  \beta_2, \beta_3 \in \Phi$ with $\beta_1+  \beta_2+ \beta_3 \in \Phi$ and $\beta_i+  \beta_j \ne 0$ for $i \ne j$. Then at least two of the three partial sums $\beta_i+  \beta_j$ belong to $ \Phi$.
\end{lemma}
\begin{proof}
 See, for example, \cite[\S11, Lemma 11.10]{LN04}.
\end{proof}

Recall the notation $\Theta^{(r)}=\{\alpha \in \Phi^+ \mid {\rm ht}(\alpha)=r\}$.
It follows from Theorem \ref{thm:dual} that $|\Theta^{(r)}|=1$  if $m_{\ell-1}<r\le m_{\ell}$.
\begin{corollary}\label{cor:G2} 
Assume that $\beta \in \Phi^+$, $\alpha \in \Delta$, $\lan\beta, \alpha \ran=3$ and $\{\alpha\}=\{\alpha_i \in \Delta \mid \beta-\alpha_i \in\Phi^+\}$. 
Then there is no $\alpha' \in \Delta\setminus\{\alpha\}$ such that $\beta-\alpha -\alpha'\in \Phi$. 
In particular, the statement holds true if the last assumption is replaced by ${\rm ht}(\beta)\ge m_{\ell-1}+2$.
\end{corollary}
\begin{proof}
Suppose that there exists $\alpha' \in \Delta\setminus\{\alpha\}$ such that $\gamma:=\beta-\alpha -\alpha'\in \Phi$. 
By Lemma \ref{lem:3roots}, $\alpha +\alpha' \in \Phi^+$. 
Thus $\lan\alpha' ,\alpha \ran \in \{ -1, -2, -3\}$.
Moreover,  $\lan\gamma ,\alpha \ran + \lan\alpha' ,\alpha \ran =1$. 
This contradicts to the fact that at most two root lengths occur in $\Phi$.
\end{proof}
Recall the notation $\U=\{\theta_i\in \Phi^+ \mid {\rm ht}(\theta_i) >m_{\ell-1}\}$, and $m=|\U|=m_{\ell}-m_{\ell-1}=m_{2}-1$.
Suppose that the elements of $\U$ are labeled so that $\theta_1$ denotes the highest root, and $\xi_i=\theta_i-\theta_{i+1} \in\Delta$ for $1 \le i \le m-1$.
We also adopt a convention $\xi_0:=-\theta_1$.
Set $\Xi:=\{\xi_i \mid 0 \le i \le m-1\}$.
Note that $\Xi$ is a multiset, not necessarily a set. 
\begin{corollary}\label{cor:2}
Suppose that $m \ge 2$.
Then the simple roots $\xi_0,\ldots,\xi_{m-2}$ all are non-ramification points of $\widetilde{\D}(\Phi)$. 
\end{corollary}
\begin{proof}
By Remark \ref{rem:only1}, the condition $m \ge 2$ ensures that $\xi_1$ is the unique vertex of ${\D}(\Phi)$ connected to $\xi_0$. 
Suppose that $m \ge 3$. 
Fix $\xi_i \in \Xi$, $1 \le i\le m-2$ and let $\alpha$ be an adjacent vertex to $\xi_i$ on $\D(\Phi)$. 
We have $(\xi_i+\alpha)+\theta_{i+1}-\alpha= \theta_i \in\Phi^+$.
By Lemma \ref{lem:3roots}, either $\theta_i+\alpha\in\Phi^+$ or $\theta_{i+1}-\alpha \in\Phi^+.$
If $i=1$ then $\alpha=\xi_2$. If $i > 1$ then $\alpha \in \{\xi_{i-1}, \xi_{i+1}\}$.
Thus $\xi_i$ is not a ramification point.
\end{proof}

\begin{lemma}\label{lem:inner-prod} 
Suppose $\beta_1, \beta_2 \in \Phi$ and $\beta_1- \beta_2\in \Phi$. 
If at least one of $\beta_1, \beta_2$ is a long root, then $(\beta_1, \beta_2)>0$.
\end{lemma}
\begin{proof}
Straightforward. 
\end{proof}

\begin{proposition}\label{prop:lengths}
Suppose that $m \ge 3$. 
\begin{enumerate}
\item[(i)] If there is an integer $t$ such that $1 \le t \le m-1$ and $\lan\theta_t,\xi_t \ran= 3$, then $t=m-2$. As a consequence, $m\ge 4$ and $\|\theta_1\|=\ldots = \|\theta_{m-2}\|=\|\xi_1\|=\ldots = \|\xi_{m-3}\|$. 
\item[(ii)] If there is no such $t$, then $\|\theta_1\|=\ldots = \|\theta_{m-1}\|=\|\xi_1\|=\ldots = \|\xi_{m-2}\|$. 
\end{enumerate}
\end{proposition}
\begin{proof}
We only give a proof for (i). 
Proof of (ii) follows from a similar argument. 
It follows from Proof of Proposition \ref{lem:coes} that $\lan\theta_1,\xi_1 \ran\in \{1,2\}$.
By Lemma \ref {rem:k=2}, we must have $\lan\theta_1,\xi_1 \ran=1$. 
Thus $\|\theta_{2}\|=\|\xi_1\|=\|\theta_{1}\|$.
The first statement of (i) follows from Lemma \ref {rem:k=2} and Corollary \ref{cor:G2}. 
One can use Lemma \ref{lem:inner-prod} to prove inductively that $\lan\theta_i,\xi_i \ran=1$ and $\theta_i,\xi_i$ all are long roots for $1 \le i \le m-3$ ($\theta_{m-2}$ is a long root as well), which proves the second statement.
\end{proof}

\vn
\emph{Convention}: 
For simplicity, in the remainder of the paper, let us call the case ``there is an integer $t$ such that $1 \le t \le m-1$ and $\lan\theta_t,\xi_t \ran= 3$" Case $1$, and its negation Case $2$.

Our candidate for a set whose cardinality is expressed in terms of $m_2$, and isomorphic (actually equal) to the set $\Lambda$ in Proposition \ref{lem:coes} will be introduced below.
For a finite multiset $S=\{(a_1)^{b_1},\ldots, (a_n)^{b_n}\}$, we write $\overline{S}$ for the base set of $S$ i.e., $\overline{S}=\{a_1,\ldots, a_n\}$.

\begin{proposition}\label{prop:b-a}\quad
\begin{enumerate}
\item[(i)] If Case 1 occurs, then $\Xi=\{\xi_0, \xi_1, \ldots, (\xi_{m-2})^2\}$ with $\xi_i \ne \xi_j$ for $0 \le i < j \le m-2$. 
As a result, $|\overline{\Xi}|=m_2-2$. 
Moreover, $(\xi_0,\xi_1,\ldots,\xi_{m-3})$ is a simple chain of $\widetilde{\D}(\Phi)$ connected to the other vertices only at $\xi_{m-3}$, and $\lan \xi_{m-3},\xi_{m-2}\ran=-3$. 
\item[(ii)] If Case 2 occurs, then  $\Xi=\{\xi_0, \xi_1, \ldots, \xi_{m-1}\}$ with $\xi_i \ne \xi_j$ for $0 \le i < j \le m-1$. 
As a result, $|\Xi|=m_2-1$. 
If $\lan\theta_1,\xi_1 \ran=2$, then $\Xi=\{\xi_1\}$ and $m=2$. 
If $\lan\theta_1,\xi_1 \ran=1$ and $m \ge 3$, then $(\xi_0,\xi_1,\ldots,\xi_{m-2})$ is a simple chain of $\widetilde{\D}(\Phi)$ connected to the other vertices only at $\xi_{m-2}$. 
\end{enumerate}
\end{proposition}
\begin{proof}
We only give a proof for (i). 
Proof of (ii) follows from a similar argument. 
Obviously, $\xi_0\ne \xi_i$ for all $1 \le i \le m-2$ by a reason of heights, and $\xi_{m-2} \ne \xi_i$ for all $0 \le i \le m-3$ by a reason of lengths.
Suppose to the contrary that $\xi_i=\xi_j$ for some $1 \le i<j \le m-3$.
Choose indexes $i,j$ so that $j-i$ is minimal.
By Proposition \ref{prop:lengths}, $j>i+1$.
If $j=i+2$, then $\theta_i=\theta_{i+3}+2\xi_{i}+\xi_{i+1}$.
This cannot happen since $\lan \theta_i,\xi_{i}\ran=1$, $\lan\theta_{i+3},\xi_{i} \ran\ge -1$ and $\lan\xi_{i+1},\xi_{i} \ran\ge -1$.
Then $j>i+2$ and
$\{\xi_i ,\xi_{i+1}\}, \{\xi_{i+1},\xi_{i+2}\}, \ldots, \{\xi_{j-1},\xi_{j}\}$ are connected subgraphs of $\D(\Phi)$. 
By the choices of $i,j$, the simple roots $\xi_i ,\xi_{i+1},\ldots,\xi_{j-1}$ are mutually distinct, the condition $\xi_i=\xi_j$ implies that $\D(\Phi)$ contains a cycle.
This contradiction proves the first statement. 
The remaining statements follow immediately. 
\end{proof}

\begin{corollary}\label{cor:differences}
 \quad
\begin{enumerate}
\item[(i)] If Case 1 occurs, then $\theta_i-\theta_j \in \Phi^+$ for $1 \le i<j \le m$, $\{i,j\} \ne \{m-2,m\}$, and $\theta_{m-2}-\theta_m \in 2\Delta$.
\item[(ii)] If Case 2 occurs, then $\theta_i-\theta_j \in \Phi^+$ for $1 \le i<j \le m$.
\end{enumerate}
\end{corollary}
\begin{proof}
We only give a proof for (i). 
Obviously, $\theta_{m-2}-\theta_m=2 \xi_{m-2} \in 2\Delta$.
By \cite[Chapter VI, \S1.6, Corollary 3(b)]{B68}, $\theta_i-\theta_{j} =\xi_i +\xi_{i+1}+\ldots+\xi_{j-1}\in \Phi^+$ for $1 \le i<j \le m-1$. 
Note that $\theta_i-\theta_m =(\theta_i-\theta_{m-1})+\xi_{m-2}$ for all $1 \le i \le m-3$. 
Thus $\theta_i-\theta_m\in\Phi^+$ because $\theta_i-\theta_{m-1}\in\Phi^+$ as above and
\begin{align*}
(\theta_i-\theta_{m-1},\xi_{m-2}) &= (\xi_i +\ldots+\xi_{m-3}+\xi_{m-2},\xi_{m-2}) \\
& =(\xi_{m-3}+\xi_{m-2},\xi_{m-2})<0.
\end{align*}
\end{proof}

\begin{remark}\label{rem:related}
Corollary \ref{cor:differences} is related to a property of a \emph{chain} in the poset in \cite[Lemma 5.1]{H16}. 
However, for the particular chain $\U$, Corollary \ref{cor:differences} is a bit more explicit and the proof does not need to go through the classification whether the root system is of type $G_2$ or not.
\end{remark}

\section{Proof of the main result}\label{sec:proof}
Theorem \ref{thm:uniform} is a consequence of the following:
 \begin{theorem}\label{thm:iso}
With the notations as in Propositions \ref{lem:coes} and \ref{prop:b-a}, $\overline{\Xi}=\Lambda$.
\end{theorem}
\begin{proof}
The proof will be proceeded in three steps. 
In what follows, $\theta$ and $\theta_1$ both denote the highest root.

Step $1$. If $c_{\max}=1$, by Remark \ref{rem:only1}, all roots of $\Phi$ have the same length.
So the problem falls in Case 2. 
It is easily seen that $\Xi=\Lambda=\{-\theta\}$, and $m_2=c_{\max}+1=2$. 
Note also that $c_{\max}=1$ if and only if $m=1$.

Step $2$. Now consider $c_{\max}\ge 2$ and $m=2$. 
This implies that $\xi_1 \equiv \lambda_1$ is the unique vertex of ${\D}(\Phi)$ connected to $-\theta$. 
By Proof of Proposition \ref{lem:coes}, $\lan\theta,\xi_1 \ran\in \{1,2\}$. 
So the problem falls in Case 2. 
Hence  $\Xi=\{-\theta, \xi_1\}$ and $m_2=m+1=3$. 
If $\lan\theta,\lambda_1 \ran=2$, by Proposition \ref{lem:coes}, $\Lambda=\{-\theta, \lambda_1\}$ and $c_{\max}=2$. 
Now consider $\lan\theta,\lambda_1 \ran=1$, and suppose that $|\Lambda| \ge 3$ i.e., $\Lambda$ contains a simple root other than $\lambda_1$, say $\lambda_2$.
Recall the notation $\Theta^{(r)}=\{\alpha \in \Phi^+ \mid {\rm ht}(\alpha)=r\}$. 
Since $m=2$, we may assume that $\Theta^{(m_{\ell-1})} \supseteq \{\mu:=\theta-\lambda_1-\lambda_2,\mu' :=\theta-\lambda_1-\lambda'_2\}$ with $\lambda'_2\ne\lambda_2$.
If $\lambda_1 =\lambda'_2$, then $\lan\mu', \lambda_1\ran =-3$, which is absurd because $\lambda_1$ is long.
If $\lambda_1 \ne\lambda'_2$, by Lemma \ref{lem:3roots},  $\lambda'_2+\lambda_1 \in \Phi^+$. 
Thus $\lambda_{1}$ is a ramification point of $\widetilde{\D}(\Phi)$, a contradiction. 
In either case, $\Xi=\Lambda=\{-\theta, \lambda_1\}$ and  $m_2=c_{\max}+1=3$. 

Step $3$.
Now consider $c_{\max}\ge 2$ and $m \ge 3$. 
The condition $m \ge 3$ ensures that $\lan\theta_1,\lambda_1\ran =1$.

Firstly, we prove that $\overline{\Xi}\supseteq\Lambda$. 
By Proposition \ref{lem:coes}(ii), we have
$$\lambda_1+ \ldots+ \lambda_i\in \Phi^+ \mbox{ and } (\theta_1,\lambda_1+ \ldots+ \lambda_i)=(\theta_1,\lambda_1)>0 \quad (1\le i \le q).$$
Set $\eta_1:=\theta_1$, and for $2 \le p \le q+1$ set 
$$\eta_p:=\theta_1 - (\lambda_1+ \ldots+ \lambda_{p-1}) \in \Phi^+, \mbox{ then } \eta_{p}=\eta_{p-1}-\lambda_{p-1}.$$
One can use Lemma \ref{lem:inner-prod} and the fact that $\lambda_1, \ldots, \lambda_{q-1}$ are long roots from Proposition \ref{lem:coes}(ii) to prove inductively that $\lan\eta_i, \lambda_i\ran =1$ for $1 \le i \le q-1$ and $\|\eta_1\|=\ldots = \|\eta_{q-1}\|= \|\eta_{q}\|$. 
We claim that 
\begin{equation}\label{eq:claim}
\Theta^{(m_l-p+1)}=\{\eta_p\} \mbox{ for } 1 \le p \le q+1.
\end{equation}
It is clearly true when $1 \le p\le2$.
Suppose to the contrary that we can choose the smallest $p$ such that $3 \le p \le q+1$ and $|\Theta^{(m_l-p+1)}|>1$. 
In particular, $m_{l-1}=m_l-p+1$.
To obtain a contradiction, we use a very similar argument to that used in Step $2$.
Assume that
$\{\eta_p,\eta'_p\} \subseteq \Theta^{(m_l-p+1)}$ with $\eta_p\ne\eta'_p$.
There exists $\lambda'_{p-1} \in \Delta$ such that $\lambda_{p-1}\ne\lambda'_{p-1}$ and 
$\eta_{p-2}=\eta'_p+\lambda'_{p-1}+\lambda_{p-2}$.
If $\lambda_{p-2} =\lambda'_{p-1}$, then $\lan\eta'_p,\lambda_{p-2}\ran =-3$, which is absurd since $\lambda_{p-2}$ is long.
If $\lambda_{p-2} \ne \lambda'_{p-1}$, by the minimality of $p$ and Lemma \ref{lem:3roots}, $\lambda'_{p-1}+\lambda_{p-2} \in \Phi^+$. 
If $p=3$, $\lambda_{1}$ is connected to three different roots: $-\theta_1, \lambda_{2}, \lambda'_{2}$, which is absurd. 
Suppose henceforth that $p \ge 4$.
Since $\lambda_{p-2}$ is connected only to $\lambda_{p-3}, \lambda_{p-1}$, we must have $\lambda_{p-3} =\lambda'_{p-1}$.
Thus $\eta_{p-3}=\eta'_p+2\lambda_{p-3}+\lambda_{p-2}$, and $\lan\eta'_p,\lambda_{p-3}\ran =-2$.
This is impossible since $\lambda_{p-3}$ is long. 
We complete the proof of the claim \eqref{eq:claim}. 
Therefore, 
$$\Lambda=\{\lambda_{p-1} \mid 2 \le p \le q+1\} =\{\eta_{p-1}-\eta_{p} \mid 2 \le p \le q+1\} \subseteq \overline{\Xi}.$$

Secondly, we prove that $\overline{\Xi}\subseteq\Lambda$. 
The proofs for Case 1 and Case 2 are similar, we only give a proof for Case 1. 
For the occurrence of Case 1, we can assume that $m \ge 4$.
We need to prove that starting from $-\theta$, going along vertices of ${\D}(\Phi)$, the elements of $\overline{\Xi}$ produce a correct path to reach a first simple root associated to $c_{\max}$.
To this end, we show that $c_{\xi_0} < c_{\xi_1} < \ldots < c_{\xi_{m-2}}$. 
Note that $\lan\xi_{m-2},\xi_{m-3}\ran=-1$ since $\xi_{m-3}$ is a long root. 
Using Proposition \ref{prop:b-a} and working out the equations $\lan\theta_1,\theta_1\ran=2, \lan\theta_1,\xi_1\ran=1, \lan\theta_1,\xi_i\ran=0$ for $2 \le i \le m-3$, we obtain
$$c_{\xi_0}=1, \, c_{\xi_1}=2, \, c_{\xi_{i+1}}= 2c_{\xi_i} -c_{\xi_{i-1}} \quad (1 \le i \le m-3).$$
Thus $c_{\xi_i}=i+ 1$ for $0 \le i \le m-2$, which proves the claim.
\end{proof}

Corollary \ref{cor:criterion}  is a consequence of the following:
  \begin{theorem}\label{thm:criterion-G2} 
The following statements are equivalent: 
\begin{enumerate}[(i)]
\item The Dynkin graph of $\Phi$ has the form 
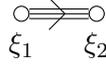
\begin{figure}[htp!]
\begin{center}
\begin{tikzpicture}[scale=1]
    \node[dnode,label=below:$\xi_{1}$] (1) at (0,0) {};
    \node[dnode,label=below:$\xi_{2}$] (2) at (1,0) {};
    \path (1) edge[tedge] (2);
\end{tikzpicture}
\end{center}
\end{figure}

\item The extended Dynkin graph of $\Phi$ has the form
\begin{figure}[htp!]
\begin{center}
\begin{tikzpicture}[scale=1]
    \node[dnode,label=below:$-\theta$] (0) at (-1,0) {};
    \node[dnode,label=below:$\xi_{1}$] (1) at (0,0) {};
    \node[dnode,label=below:$\xi_{2}$] (2) at (1,0) {};
    \path (0) edge  (1) ;
    \path (1) edge[tedge] (2);
\end{tikzpicture}
\end{center}
\end{figure}

\item $c_{\max}=m_2-2$.
\end{enumerate}

\end{theorem}
\begin{proof}
(i) $\Leftrightarrow$ (ii) is clear. 
(ii) $\Rightarrow$ (iii) 
It is easy to calculate from the graph that $c_{\max}=c_{\xi_2}=3$. 
By Theorem \ref{thm:uniform}, $m_2 \ge 4$. 
Note that $\theta_2 := \theta- \xi_1$ is the unique root of height $m_\ell -1$. 
Moreover, $\lan \theta_2, \xi_2\ran=-\lan \xi_1,\xi_2\ran=3$. 
So the problem falls in Case 1. 
Thus $c_{\max}=m_2-2$.
(ii) $\Leftarrow$ (iii) 
The condition $c_{\max}=m_2-2$ ensures that $c_{\max}\ge 2$ i.e., $m_2 \ge 4$. 
Theorem \ref{thm:uniform} implies that the problem falls in Case 1. 
In particular, $m_2 \ge 5$ by Proposition \ref{prop:lengths}(i). 
Moreover, by Proposition \ref{lem:coes} and Theorem \ref{thm:iso}, the equations $\lan\theta, \xi_{m-2}\ran=0$, $c_{\xi_{m-2}}=c_{\max}$, $c_{\xi_{m-3}}=c_{\max}-1$ yield $c_{\max} \le 3$.
So we must have $c_{\max}= 3$ i.e., $m_2=5$, and there are no vertices of ${\D}(\Phi)$ connected to $\xi_2$ other than $\xi_1$. 
This completes the proof.

 \end{proof}

\section*{Acknowledgements}
This paper originates from the author's Master's thesis, 
written under the supervision of Professor Hiroaki Terao at the Hokkaido University. 
The author wishes to express his sincere thanks to Professor Terao for many stimulating conversations.
The author also gratefully acknowledges the support of  Master's scholarship program of the Japanese Ministry of Education, Culture, Sports, Science, and Technology (MEXT) under grant number 142506.

\bibliographystyle{alpha} 
\bibliography{references}

\end{document}